\documentclass[11pt,oneside, a4paper]{amsart} 

\usepackage[english]{babel}
\usepackage[T1]{fontenc}
\usepackage{amsmath}
\usepackage{amssymb}
\usepackage{float}
\usepackage{amsfonts}
\usepackage{mathtools}
\usepackage[utf8]{inputenc}
\usepackage[mathcal]{eucal}
\usepackage{enumerate}
\usepackage{amsthm}
\usepackage{hyperref}
\usepackage[totalwidth=15cm,totalheight=21cm]{geometry}
\usepackage{epsfig,fancyhdr,color}
\usepackage{multicol} 

\newtheorem{cor}{Corollary}[section]

\newtheorem{teo}[cor]{Theorem}

\newtheorem{prop}[cor]{Proposition}

\newtheorem{lemma}[cor]{Lemma}

\theoremstyle{definition}

\theoremstyle{remark}

\newtheorem*{remark*}{Remark}

\newcommand{\Pp}{\mathbb{P}}
\newcommand{\R}{\mathbb{R}}

\newcommand{\C}{\mathbb{C}}

\newcommand{\h}{\mathbb{H}}

\newcommand{\diag}{\mathrm{diag}}

\newcommand{\dPSL}{\mathbb{P}SL(2,\mathbb{R})\times \mathbb{P}SL(2,\mathbb{R})}
\newcommand{\PSL}{\mathbb{P}SL}

\newcommand{\SO}{\mathrm{SO}}

\newcommand{\AdS}{\mathrm{AdS}}
\newcommand{\Imm}{\mathcal{I}\text{m}}
\newcommand{\Teich}{\mathpzc{T}}
\renewcommand{\Re}{\mathcal{R}\text{e}}

\DeclareMathAlphabet{\mathpzc}{OT1}{pzc}{m}{it}

\title[Degeneration GHM AdS structures]{Degeneration of globally hyperbolic maximal \\ anti-de Sitter structures along pinching sequences}

\author{Andrea Tamburelli}

\begin{document}

\begin{abstract}
Let $S$ be a closed oriented surface of genus at least $2$. Using the parameterisation of the deformation space of globally hyperbolic maximal anti-de Sitter structures on $S \times \R$ by the cotangent bundle over the Teichm\"uller space of $S$, we study the behaviour of these geometric structures along pinching sequences. We show, in particular, that the regular globally hyperbolic anti-de Sitter structures introduced in \cite{Tambu_regularAdS} naturally appear as limiting points.
\end{abstract}

\maketitle
\section*{Introduction}
A quickly growing area of reaserch studies geometric structures associated to surface groups representations into Lie groups, with the aim of understanding to which extent the well-known Teichm\"uller theory for representations into $\PSL(2,\R)$ can be generalised to higher rank Lie groups (\cite{Wienhard_ICM}). Globally Hyperbolic Maximal (GHM) anti-de Sitter structures on $S\times \R$ correspond in this context to pairs of faithful and discrete representations into $\PSL(2,\R)$. It turns out that these manifolds share many similarities with hyperbolic quasi-Fuchsian manifolds (\cite{Schlenker-Krasnov}, \cite{Mess}), and their study has led also to a better understanding of canonical maps between hyperbolic surfaces, for instance earthquakes (\cite{bsk_multiblack},\cite{BonSchlGAFA2009}) and minimal Lagrangian diffeomorphisms (\cite{bon_schl},\cite{seppimaximal}, \cite{Tambu_poly}). \\
When $S$ is closed, Krasnov and Schlenker (\cite{Schlenker-Krasnov}) found another parameterisation of the deformation space of GHMC anti-de Sitter structures by the cotangent bundle over the Teichm\"uller space of $S$: a point $(h,q)\in T^{*}\Teich(S)$ is associated to the GHM anti-de Sitter manifold containing an embedded space-like maximal surface (i.e. with vanishing mean curvature) with induced metric conformal to $h$ and second fundamental form given by $2\Re(q)$. This construction has been later generalised by the author to the case of surfaces with punctures $S_{\mathpzc{p}}=S\setminus \{p_{1}, \dots, p_{k}\}$, first by allowing second order pole singularities at the punctures (\cite{Tambu_regularAdS}) and then higher order poles (\cite{Tambu_wildAdS}). We remark that, unlike the closed case, the holonomy representation does not determine completely the structure, and new extra data describing the boundary curve at infinity of the maximal surface are necessary. \\

Now, holomorphic quadratic differentials with second order poles naturally appear as limits of holomorphic quadratic differentials over closed Riemann surfaces that are degenerating by pinching disjoint simple closed curves (\cite{Wolpert_spectrallimit}. Thus, one may wonder if regular globally hyperbolic anti-de Sitter structures are somehow limits of degenerating sequences of GHMC anti-de Sitter structures. We prove the following: \\

\noindent \textbf{Theorem A} \textit{Let $M_{n}$ be a sequence of GHMC anti-de Sitter manifolds parameterised by points $(h_{n},q_{n})\in T^{*}\Teich(S)$ such that $h_{n}$ diverges to a complete hyperbolic metric $h_{\infty}$ on the regular part of a noded Riemann surface $\Sigma^{reg}=\Sigma\setminus \{\text{nodes}\}$ and $q_{n}$ converges to a regular quadratic differential $q_{\infty}$ over $\Sigma^{reg}$. Then the sequence $M_{n}$ converges to the regular GHM anti-de Sitter structure parameterised by $(h_{\infty}, q_{\infty})$.}\\

The convergence in Theorem A has to be intended as follows. Let $\{\gamma_{j}\}_{j=1, \dots, k}$ be the simple closed curves that are getting pinched for the sequence $h_{n}$. Let $S^{reg}=S\setminus \cup_{j}\gamma_{j}$. For any connected component $S_{c}^{reg}$ of $S^{reg}$ and $\Sigma^{reg}_{c}$ of $\Sigma^{reg}$ choose a converging sequence of base points. We will prove that the equivariant maximal embeddings $\tilde{\sigma}_{n}$ with embedding data $(h_{n}, q_{n})$ restricted to the universal cover of $S^{reg}_{c}$ converge in the pointed topology to the restriction to the universal cover of $\Sigma_{c}^{reg}$ of the maximal embedding $\tilde{\sigma}_{\infty}$ with embedding data $(h_{\infty}, q_{\infty})$. From this we will deduce the convergence of the developing maps in each connected component and of the holonomy representations. Therefore, one should interpret each connected component of $M_{\infty}$ as pointed geometric limit of $M_{n}$ as the point is chosen in a fixed connected component of the regular part of the maximal surface, and, consequently,  $M_{\infty}$ as the collection of all possible such limits. \\

We remark that a similar picture holds in the context of convex real projective structures (\cite{Loftin_neck}), where now quadratic differentials are replaced by cubic differentials, and the role of the maximal surfaces is played by hyperbolic affine spheres in $\R^{3}$. It would be interesting to understand if this phenomenon is shared by other geometric structures.\\

The above result should also be interpreted as complementary to the study of degenerations of GHMC anti-de Sitter structures along rays of quadratic differentials carried out in \cite{entropy}. 

\subsection*{Outline of the paper} In Section \ref{sec:background} we recall well-known results about anti-de Sitter geometry, maximal globally hyperbolic manifolds and holomorphic quadratic differentials. Section \ref{sec:main} is then devoted to the proof of the main theorem.

\section{Background material}\label{sec:background}

We recall here some well-known facts about anti-de Sitter geometry and Wolpert's plumbing coordinates that will be used in the sequel. Throughout the paper, we will denote with $S$ a closed, connected, oriented surface of genus at least $2$ and with $\mathcal{T}(S)$ the Teichm\"uller space of $S$, which will be thought of, according to the context, as both the space of hyperbolic metrics on $S$ up to isometries isotopic to the identity and the space of complex structures on $S$ up to diffeomorphisms isotopic to the identity. The moduli space of Riemann surfaces homeomorphic to $S$ will be denoted by $\mathcal{M}(S)=\mathcal{T}(S)/\mathrm{MCG(S)}$.

\subsection{Anti-de Sitter geometry}\label{subsec:background} Consider the vector space $\R^{4}$ endowed with a bilinear form of signature $(2,2)$
\[
	\langle x,y\rangle= x_{0}y_{0}+x_{1}y_{1}-x_{2}y_{2}-x_{3}y_{3} \ .
\]
We define
\[
	\widehat{\AdS}_{3}=\{ x \in \R^{4} \ | \ \langle x , x \rangle=-1 \} \ .
\]
It can be easily verified that $\widehat{\AdS}_{3}$ is diffeomorphic to a solid torus and the restriction of the bilinear form to the tangent space at each point endows $\widehat{\AdS}_{3}$ with a Lorentzian metric of constant sectional curvature $-1$. Anti-de Sitter space is then 
\[
	\AdS_{3}=\Pp(\{x \in \R^{4} \ | \ \langle x,x\rangle < 0\})\subset \R\Pp^{3} \ .
\]
The natural map $\pi:\widehat{\AdS}_{3} \rightarrow \AdS_{3}$ is a two-sheeted covering and we endow $\AdS_{3}$ with the induced Lorentzian structure. The isometry group of $\widehat{\AdS_{3}}$ that preserves the orientation and the time-orientation is $\SO_{0}(2,2)$, the connected component of the identity of the group of linear transformations that preserve the bilinear form of signature $(2,2)$. \\

The boundary at infinity of anti-de Sitter space is naturally identified with 
\[
	\partial_{\infty}\AdS_{3}=\Pp(\{ x \in \R^{4} \ | \ \langle x,x\rangle=0\}) \ .
\]
It coincides with the image of the Segre embedding $s:\R\Pp^{1}\times \R\Pp^{1} \rightarrow \R\Pp^{3}$, and thus it is foliated by two families of projective lines. The action of an isometry extends continuously to the boundary, and preserves the two foliations. Moreover, it acts on each line by a projective transformation, thus giving an identification between $\Pp\SO_{0}(2,2)$ and $\dPSL$. \\

The Lorentzian metric on $\AdS_{3}$ induces on $\partial_{\infty}\AdS_{3}$ a conformally flat Lorentzian structure so that the light-cone at each point $p \in \partial_{\infty}\AdS_{3}$ is generated by the two lines in the foliation passing through $p$. 

\subsubsection{Complete maximal surfaces in $\AdS_{3}$} \label{sec:maxAdS} Let $U\subset \h^{2}$ be a simply connected domain. We say that $\sigma:U \rightarrow \AdS_{3}$ is a space-like embedding if $\sigma$ is an embedding and the induced metric $I=\sigma^{*}g_{AdS}$ is Riemannian. The Fundamental Theorem of surfaces embedded in anti-de Sitter space ensures that such a space-like embedding is uniquely determined, up to post-composition by a global isometry of $\AdS_{3}$, by its induced metric $I$ and its shape operator $B:\sigma_{*}TU \rightarrow \sigma_{*}TU$, which satisfy
\[
	\begin{cases} 
		d^{\nabla}B=0 \ \ \ \ \ \ \ \ \ \ \  \ \ \ \ \ \ \ \ \ \ \ \ \ \ \ \text{(Codazzi equation)} \\
		K_{I}=-1-\det(B) \ \ \ \ \ \ \ \ \ \ \ \ \ \text{(Gauss equation)}
	\end{cases}
\]
where $\nabla$ is the Levi-Civita connection and $K_{I}$ is the curvature of the induced metric on $\sigma(U)$. \\

We say that $\sigma$ is a maximal embedding if $B$ is traceless. In this case, the Codazzi equation implies that the second fundamental form $II=I(B\cdot, \cdot)$ is the real part of a quadratic differential $q$, which is holomorphic for the complex structure compatible with the induced metric $I$, in the following sense. For every pair of vector fields $X$ and $Y$ on $\sigma(U)$, we have
\[
	\Re(q)(X,Y)=I(BX,Y) \ .
\]
In a local conformal coordinate $z$, we can write $q=f(z)dz^{2}$ with $f$ holomorphic and $I=e^{2u}|dz|^{2}$. Thus, $\Re(q)$ is the bilinear form that in the frame $\{\partial_{x}, \partial_{y}\}$ is represented by 
\[
	\Re(q)=\begin{pmatrix}
			\Re(f) & -\Imm(f) \\
			-\Imm(f) & -\Re(f)
		\end{pmatrix} \ ,
\]
and the shape operator can be recovered as $B=I^{-1}\Re(q)$.\\

If the induced metric is complete, the space-like condition implies that, identifying $\widehat{AdS}_{3}$ with $D\times S^{1}$, the surface is the graph of a $2$-Lipschitz map (\cite[Proposition 3.1]{Tambu_poly}) and its boundary at infinity $\Gamma$ is a locally achronal topological circle in $\partial_{\infty}\AdS_{3}$ with the following property: if two points are causally related, then the light-like segment joining them is entirely contained in $\Gamma$ (\cite[Corollary 3.3]{Tambu_poly}).

\subsubsection{Globally hyperbolic anti-de Sitter structures} This paper deals with the moduli space of a special class of manifolds locally isometric to $\AdS_{3}$. \\

We say that an anti-de Sitter three-manifold $M$ is Globally Hyperbolic Maximal (GHM) if it contains an embedded, oriented, space-like surface $S$ that intersects every inextensible non-space-like curve in exactly one point, and if $M$ is maximal by isometric embeddings. It turns out that $M$ is necessarily diffeomorphic to a product $S\times \R$ (\cite{MR0270697}). Moreover, we say that $M$ is Cauchy Compact (C) if $S$ is closed of genus at least $2$. We denote with $\mathcal{GH}(S)$ the deformation space of GHMC anti-de Sitter structures on $S\times \R$. 

\begin{teo}[\cite{Schlenker-Krasnov}] The deformation space of GHMC anti-de Sitter structures is parameterised by the cotangent bundle of the Teichm\"uller space of $S$. 
\end{teo}

Let us recall briefly how this homeomorphism is constructed. Let $M$ be a GHMC anti-de Sitter manifold. It is well-known that $M$ contains a unique embedded maximal surface $S$ (\cite{foliationCMC}). Lifting $S$ to $\AdS_{3}$, we obtain an equivariant maximal embedding of $\h^{2}$ into $\AdS_{3}$, which is completely determined (up to global isometries of $\AdS_{3}$) by its induced metric and a holomorphic quadratic differential. By equivariance, these define a Riemannian metric $I$ and a holomorphic quadratic differential $q$ on $S$. We can thus define a map
\begin{align*}
	\Psi: \mathcal{GH}(S) &\rightarrow T^{*}\Teich(S) \\
			M &\mapsto (h,q)
\end{align*}
associating to a GHMC anti-de Sitter structure the unique hyperbolic metric in the conformal class of $I$ and the holomorphic quadratic differential $q$. In order to prove that $\Psi$ is a homeomorphism, Krasnov and Schlenker (\cite{Schlenker-Krasnov}) found an explicit inverse. They showed that, given a hyperbolic metric $h$ and a quadratic differential $q$ that is holomorphic for the complex structure compatible with $h$, it is always possible to find a smooth map $v:S\rightarrow \R$ such that $I=2e^{2v}h$ and $B=I^{-1}\Re(2q)$ are the induced metric and the shape operator of a maximal surface embedded in a GHMC anti-de Sitter manifold. This is accomplished by noticing that the Codazzi equation for $B$ is trivially satisfied since $q$ is holomorphic, and thus it is sufficient to find $v$ so that the Gauss equation holds. Now,
\[
	\det(B)=\det(e^{-2v}(2h)^{-1}\Re(q))=e^{-4v}\det((2h)^{-1}\Re(2q))=-e^{-4v}\|q\|_{h}^{2}
\]
and
\[
	K_{I}=e^{-2v}(K_{2h}-\Delta_{2h}v)=\frac{1}{2}e^{-2v}(K_{h}-\Delta_{h}v)
\]
hence the Gauss equation translates into the quasi-linear PDE
\begin{equation}\label{eq:PDE}
	\frac{1}{2}\Delta_{h}v=e^{2v}-e^{-2v}\|q\|_{h}^{2}+\frac{1}{2}K_{h} \ .
\end{equation}

They proved existence and uniqueness of the solution to Equation (\ref{eq:PDE}) on closed surfaces and on surfaces with punctures, when $q$ has pole sigularities of order at most $1$ at the punctures. This construction has been later generalised by the author to the case of pole singularities of order at most $2$ obtaining the following:

\begin{teo}[\cite{Tambu_regularAdS}]\label{teo:AdS regular}Let $S_{\mathpzc{p}}=S\setminus\{p_{1}, \dots, p_{n}\}$ be a surface with punctures and negative Euler characteristic. Given a complete hyperbolic metric $h$ on $S_{\mathpzc{p}}$ and a meromorphic quadratic differential $q$ with poles of order at most $2$ at the punctures, there exists a unique complete equivariant maximal surface with induced metric conformal to $h$ and second fundamental form equal to $2\Re(q)$. 
\end{teo}

Starting then from the equivariant maximal embedding into $\AdS_{3}$, it is possible to construct a maximal domain of discontinuity for the holonomy representation, thus obtaining the desired globally hyperbolic anti-de Sitter manifold as a quotient. The manifolds obtained from Theorem \ref{teo:AdS regular} are called \emph{regular} (in order to distinguish them from the wild analogues in \cite{Tambu_wildAdS}) and their deformation space $\mathcal{GH}^{reg}(S_{\mathpzc{p}})$ is thus parameterised by the bundle over Teichm\"uller space of $S_{\mathpzc{p}}$ of meromorphic quadratic differentials with poles of order at most $2$ at the punctures. \\

The aim of this paper is to explain how the two deformation spaces $\mathcal{GH}(S)$ and $\mathcal{GH}^{reg}(S_{\mathpzc{p}})$ interact with each other by showing that $\mathcal{GH}^{reg}(S_{\mathpzc{p}})$ naturally appears when the conformal structure of the maximal surface gets pinched. To this aim, we will make use of Wolpert's plumbing coordinates, that we recall in the next section. 

\subsection{Plumbing coordinates} We recall here well-known facts about the topology of the bundle $\mathcal{V}(S)$ of regular quadratic differentials over $\overline{\mathcal{M}(S)}$ using Wolpert's plumbing coordinates. A good reference for the material covered in this section is \cite{Wolpert_familyRS}.\\

Let $\Sigma$ be a noded Riemann surface with $n$ nodes. We will denote with $\Sigma^{reg}$ the (possibly disconnected) punctured Riemann surface obtained by removing the nodes. We think of $\Sigma$ as a point in the Deligne-Mumford compactification of the moduli space of complex structures on $S$. For each node there is a cusp neighbourhood $N_{i}$ so that:
\begin{itemize}
	\item $\overline{N}_{i} \cap \overline{N}_{j} = \emptyset$ if $i\neq j$;
	\item there are coordinates $z_{i}$ and $w_{i}$ on the two connected components of $N_{i}^{reg}=\Sigma^{reg} \cap N_{i}$ and a uniform constant $c<1$ so that
\[
	N_{i}^{reg}=\{ z_{i} \ | \ |z_{i}| \in (0,c) \} \cup \{ w_{i} \ | \ |w_{i}|\in (0,c)\}
\]
and the complete hyperbolic metric $h$ on $\Sigma^{reg}$ restricts to $N_{i}^{reg}$ as 
\[
	{h}_{|_{N_{i}^{reg}}}=\frac{|dx|^{2}}{(|x|\log|x|)^{2}} \ \ \ \ \ x=z_{i},w_{i} \ .
\]
\end{itemize}
Moreover, Wolpert (\cite{Wolpert_spectrallimit}) constructed a real analytic family of smooth Beltrami differentials $\nu(s)$ on $\Sigma^{reg}$ for $s$ in a neighbourhood of the origin of $\C^{3g-3-n}$ so that
\begin{itemize}
	\item $\nu(0)=0$;
	\item the support of each $\nu(s)$ is disjoint from the closure of each $N_{i}$;
	\item there is a Riemann surface $\Sigma^{s}$ and a diffeomorphism $\chi^{s}:\Sigma^{reg} \rightarrow \Sigma^{s,reg}$ satisfying $\overline{\partial}\chi^{s}=\nu(s)\partial\chi^{s}$;
	\item the restriction of $\chi^{s}$ on each $N_{i}$ is a rotation.
\end{itemize}
The plumbing construction produces from $\Sigma$ a new Riemann surface $\Sigma^{t}$ by opening the node. Recall that $N_{i}$ is biholomorphic to the set $\{ (z_{i}, w_{i}) \in \C^{2} \ | \ z_{i}w_{i}=0, \ |z_{i}|<c, \ |w_{i}|<c\}$. For each $t_{i}\in \C$ with $|t_{i}|<c^{2}$, the surface $\Sigma^{t}$, where $t=(t_{1}, \dots, t_{n})$, is obtained by replacing each $N_{i}$ with the annulus
\[
	N_{i}^{t_{i}}=\{(z_{i},w_{i}) \in \C^{2} \ | \  z_{i}w_{i}=t_{i}, \ |z_{i}|, |w_{i}| \in (|t_{i}|/c, c)\} \ .
\]
Since the Beltrami differentials are constructed so that the hyperbolic cusps are essentially preserved, the pair $(s,t)$ as above gives a coordinate system in a neighbourhood of $\Sigma \in \overline{\mathcal{M}(S)}$. On each of these Riemann surfaces $\Sigma^{s,t}$, Wolpert defined a grafting metric $g^{s,t}$. For our purposes, we only need to know that $g^{s,t}$ is a complete metric on $\Sigma^{s,t,reg}$ in the conformal class of the unique complete hyperbolic metric $h^{s,t}$ on $\Sigma^{s,t,reg}$ and coincides with it on $N_{i}$ if $t_{i}=0$. \\

Morover, we will also need a perturbation of the grafting metric described in (\cite{Loftin_neck}), and whose construction we sketch in the next paragraph. Since $\overline{\mathcal{M}(S)}$ is compact, we can find a finite number of coordinate charts $\{V^{\alpha}\}_{\alpha=0, \dots, N}$ so that $V^{0}$ entirely lies in the thick part of the moduli space for some $\epsilon>0$ and each other $V^{\alpha}$ is a plumbing coordinate neighbourhood of a noded Riemann surface. Let $\overline{\mathcal{T}(S)}$ be the augmented Teichm\"uller space and $\pi:\overline{\Teich(S)} \rightarrow \overline{\mathcal{M}(S)}$ be the canonical projection. For each (possibly noded) marked Riemann surface $\Sigma^{\alpha,s,t} \in \pi^{-1}(V^{\alpha})$ we can define a metric $m^{\alpha, s, t}$ with the following properties:
\begin{enumerate}
	\item if $\alpha=0$, $m^{0}$ is the complete marked hyperbolic metric;
	\item otherwise, we set $m^{\alpha, s,t}$ to be equal to the marked plumbing metric $g^{s,t}$ on the complement of $N_{i}$;
	\item if $t_{i}=0$, denote with $\ell=\log(x)$ for $x=z_{i}, w_{i}$ and define
\[
		m^{\alpha, s, t}_{|_{N_{i}}}=\begin{cases}
							\begin{tabular}{l l l}
						         $(2\log(c))^{-2}|d\ell|^{2}$ & \ \ \ \ \ \ \ \ & \text{if $\Re(\ell)\leq 2\log(c)$} \\
							$g^{s,t}$ & \ \ \ \ \ \ \ \ &\text{if $\Re(\ell)\geq \log(c)$}\\
							$fg^{s,t}$ & \ \ \ \ \ \ \ \ & \text{otherwise}
							\end{tabular}
							\end{cases}
\]
		where $f$ is a smooth interpolating function of $\Re(\ell)$;
	\item if $t_{i}\neq 0$, we define 
\[
		m^{\alpha, s, t}_{|_{N_{i}}}=\begin{cases}
							\begin{tabular}{l l l}
						         $(2\log(c))^{-2}|d\ell|^{2}$& \ \ \ \ \ & \text{if $\Re(\ell) \in [\log|t_{i}|-K,K]$}\vspace{+0.2cm} \\
							 $g^{s,t}$ & \ \ \  \ \ &\text{if $|t_{i}|\geq c^{2\pi}$ or $\Re(\ell)\geq K-\log(c)$}  \\ 
							\ & \ & \text{or $\Re(\ell)\leq \log|t_{i}|-K+\log(c)$} \vspace{+0.1cm} \\ 
							$\tilde{f}g^{s,t}$ & \ \ \ \ \ & \text{otherwise}
							\end{tabular}
							\end{cases}
\]
		where $\tilde{f}$ is an interpolating function of $\Re(\ell)$ and 
\[
	K=\frac{\log|t_{i}|}{\pi}\arcsin\left(\frac{2\pi\log(c)}{\log(|t_{i}|)}\right)   \ \ \ \  \ \text{with $0<|t_{i}|<c^{2\pi}$} \ . 
\]
\end{enumerate} 

A regular quadratic differential over a noded Riemann surface $\Sigma$ is given by a holomorphic quadratic differential $q$ on $\Sigma^{reg}$ with a precise behaviour at the node: we require that in the $z_{i}$ and $w_{i}$ coordinates, $q$ has a pole of order at least $2$ at the origin and the complex residue (i.e. the coefficient of order $-2$ in the Laurent expansion of $q$) match up. Let $\mathcal{V}(S)$ be the bundle of regular quadratic differentials over $\overline{\mathcal{T}(S)}$. We say that a sequence  $(h_{n},q_{n})$ converges to $(h_{\infty}, q_{\infty})$ in $\mathcal{V}(S)$ if and only if:  
\begin{enumerate}[i)]
	\item if $h_{\infty} \in \Teich(S)$, then $(h_{n}, q_{n})$ converges to $(h_{\infty}, q_{\infty})$ in $T^{*}\Teich(S)$;
	\item if $h_{\infty}$ is a complete marked hyperbolic metric on the regular part of a noded Riemann surface $\Sigma$, we have that
	\begin{enumerate}[a)]
		\item $h_{n}$ converges to $h_{\infty}$ uniformly on compact sets on each connected component of $S^{reg}$ and $\Sigma^{reg}$ (\cite{Mondello_augmented});
		\item $q_{n}$ converges to $q_{\infty}$ in $L^{\infty}$ with respect to the hyperbolic metrics $h_{n}$;
		\item on each $N_{i}$, if we write $q_{n}=\hat{q}_{n}dx^{2}$ for $x=z_{i},w_{i}$ and $q_{\infty}=\hat{q}_{\infty}dx^{2}$, we have that $\hat{q}_{n}$ converges to $\hat{q}_{\infty}$ normally for $|x| \in (|t_{i}|/c,c)$.
	\end{enumerate}
\end{enumerate}
We remark that, by a result of Loftin (\cite[Lemma 2.8.2]{Loftin_neck}) the convergence of the quadratic differentials is in $L^{\infty}$ also with respect to the metrics $m^{\alpha,n}$.

\section{Limit along pinching sequences}\label{sec:main}
We have now all the ingredients to prove the main result. Let $(h_{n}, q_{n}) \in T^{*}\Teich(S)$ be a sequence converging to $(h_{\infty},q_{\infty}) \in \mathcal{V}(S)$ as explained above. Let $\{\gamma_{j}\}_{j=1, \dots, k}$ be the family of simple closed curves that are getting pinched by the metrics $h_{n}$. We denote with $S^{reg}=S\setminus \cup_{j}\gamma_{j}$ the possibly disconnected surface obtained by removing the curves $\gamma_{j}$ and with $\Sigma^{reg}$ the regular part of the limiting noded Riemann surface. In each connected component $S^{reg}_{i}$ of $S^{reg}$ fix a base point $p_{i}$. Associated to these data, we have equivariant maximal embeddings $\tilde{\sigma}_{i,n}:(\widetilde{S^{reg}_{i}}, \tilde{h}_{n}) \rightarrow \AdS_{3}$ and $\tilde{\sigma}_{i,\infty}:(\widetilde{\Sigma^{reg}_{i}}, h_{\infty}) \rightarrow \AdS_{3}$. 

\begin{lemma}\label{lm:metrics} The sequence of induced metrics $I_{n}$ converges smoothly on compact sets of $S^{reg}$ to $I_{\infty}$. 
\end{lemma}
\begin{proof} To enlighten the notation, we remove the dependence on the connected component, but all the computation below is to be intended to hold on each connected component of $S^{reg}$. We can assume that for $n$ large enough, the sequence $(S,h_{n})$ is entirely contained in a chart $\pi^{-1}(V^{\alpha})$ for some $\alpha \neq 0$. In order to prove the lemma, we use the method of sub- and super-solutions. By the discussion of Section \ref{subsec:background}, we can write $I_{n}=2e^{v_{n}}h_{n}$, where $v_{n}$ is the solution to the PDE
\[
	\frac{1}{2}\Delta_{h_{n}}v_{n}=e^{2v_{n}}-e^{-2v_{n}}\|q_{n}\|_{h_{n}}^{2}+\frac{1}{2}K_{h_{n}}=F(v_{n}, q_{n}) \ .
\]
Noticing that $F(-\log(\sqrt{2}), q_{n})\leq 0$, we deduce that $-\log(\sqrt{2})$ is a sub-solution, hence $v_{n}\geq -\log(\sqrt{2})$. In order to construct a family of supersolutions $V_{n}$, we consider the conformal factors $\phi_{n}$ relating the metrics $m_{n}$ and $h_{n}$, so that $m_{n}=e^{2\phi_{n}}h_{n}$. We claim that there is a constant $C$ such that $V_{n}=\phi_{n}+C$ is a super-solution. Infact, 
\begin{align*}
	&\frac{1}{2}\Delta_{h_{n}}V_{n}-e^{2V_{n}}+e^{-2V_{n}}\|q_{n}\|_{h_{n}}^{2}-\frac{1}{2}K_{h_{n}} \\
	=&\frac{1}{2}\Delta_{h_{n}}\phi_{n}-e^{2\phi_{n}}e^{2C}+e^{-2\phi_{n}}e^{-2C}\|q_{n}\|_{h_{n}}^{2}-\frac{1}{2}K_{h_{n}} \\
	=&\frac{1}{2}(\Delta_{h_{n}}\phi_{n}-K_{h_{n}})-e^{2\phi_{n}}(e^{2C}-e^{-2C}\|q_{n}\|_{m_{n}}^{2})\\
	=&-e^{2\phi_{n}}(K_{m_{n}}+e^{2C}-e^{-2C}\|q_{n}\|_{m_{n}}^{2})
\end{align*}
is negative if $C$ is sufficently big, because by definition of the metrics $m_{n}$, the function $\phi_{n}$ has a uniform lower bound, the curvature of the metrics $m_{n}$ is uniformly bounded and by \cite[Lemma 2.8.2]{Loftin_neck}, the norms $\|q_{n}\|_{m_{n}}^{2}$ are uniformly bounded. Notice that $V_{n}$ is a smooth function on each connected component of $S^{reg}$ and we have uniform bound of compact sets because $\phi_{n}$ varies continously in Wolpert's plumbing coordinates, at least for $|t|$ sufficiently small. Thus the solution to Equation (\ref{eq:PDE}) satisfies $-\log(\sqrt{2}) \leq v_{n} \leq V_{n}$. The uniform bounds on $V_{n}$ and the convergence of the background metrics $h_{n}$ to $h_{\infty}$ on compact sets imply, together with interior elliptic estimates, that the sequence $v_{n}$ is locally uniformly bounded in $C^{2,\alpha}$. This implies that, up to sub-sequences, the sequence $v_{n}$ converges to $v_{\infty}$, which solves the differential equation
\[
	\frac{1}{2}\Delta_{h_{\infty}}v_{\infty}=e^{2v_{\infty}}-e^{-2v_{\infty}}\|q_{\infty}\|_{h_{\infty}}^{2}+\frac{1}{2}K_{h_{\infty}} \ .
\]
Since $v_{n}\geq -\log(\sqrt{2})$, also $v_{\infty}$ has the same lower-bound, and, in particular,
\[
	2e^{2v_{\infty}}h_{\infty} \geq h_{\infty} \ ,
\]
which ensures completeness of the limiting metric. By the uniqueness in Theorem \ref{teo:AdS regular}, we deduce that $I_{\infty}=2e^{2v_{\infty}}h_{\infty}$, and that every subsequence of $v_{n}$ converges to the same limit $v_{\infty}$, thus showing that $I_{n}$ converges to $I_{\infty}$ smoothly on compact sets.
\end{proof}

\begin{prop}\label{prop:embeddings}The sequence of embeddings $\tilde{\sigma}_{n,i}$ converges smoothly on compact set to $\tilde{\sigma}_{\infty,i}$ up to global isometries.
\end{prop}
\begin{proof} By the fundamental theorem of surfaces embedded in anti-de Sitter space, it is sufficient to prove that the sequences of induced metrics $I_{n}$ and second fundamental forms $II_{n}$ converge smoothly on compact set to $I_{\infty}$ and $II_{\infty}$. This follows from the previous lemma and the fact that 
\[
	II_{n}=2\Re(q_{n}) \to \Re(q_{\infty})=II_{\infty}
\] 
smoothly on compact sets by definition of the topology in $\mathcal{V}(S)$ (see condition ii) point b) at the end of the previous section).
\end{proof}

The embeddings $\tilde{\sigma}_{n,i}$ come together with representations $\rho_{n,i}:\pi_{1}(S^{reg}_{i}, p_{i}) \rightarrow \SO_{0}(2,2)$ so that 
\[
	\tilde{\sigma}_{i,n}(\alpha\cdot \tilde{p})=\rho_{n}(\alpha)\tilde{\sigma}_{i,n}(\tilde{p})   \ \ \ \ \  \forall \alpha \in\pi_{1}(S^{reg}_{i}) \ \ \forall \tilde{p}\in \tilde{S}^{reg}_{i} \ .
\]
Once the embeddings are known, the representations can be computed by using the techiques of moving frames. It is convenient to use complex coordinates, so we will consider $\R^{4}\subset \C^{4}$ and extend the $\R$-bilinear form of signature $(2,2)$ to the hermitian product on $\C^{4}$ given by
\[
	\langle z,w\rangle=z_{1}\bar{w}_{1}+z_{2}\bar{w}_{2}-z_{3}\bar{w}_{3}-z_{4}\bar{w}_{4} \ .
\]
Given an equivariant maximal conformal embedding $\tilde{\sigma}: \mathcal{D} \subset \h^{2} \rightarrow \AdS_{3}$, we define the frame field of $\tilde{\sigma}$ as the map $F:\mathcal{D} \rightarrow \SO_{0}(2,2)$ which associates to any point $z\in \h^{2}$ the matrix whose columns are given by the vectors $v_{1}(z)=\tilde{\sigma}_{z}(z)/\|\tilde{\sigma}_{z}(z)\|, v_{2}(z)=\tilde{\sigma}_{\bar{z}}(z)/\|\tilde{\sigma}_{\bar{z}}(z)\|, N(z)$ and $\tilde{\sigma}(z)$, where $N(z)$ is the unit normal vector. Taking the derivatives of the fundamental relations between these vectors, one can find that $F$ satisfies the following ODE, involving only the conformal factor $\varphi$ of the induced metric $I=2e^{2\varphi}|dz|^{2}$ and the quadratic differential $q$ 
\begin{equation}\label{eq:framefield}
	F^{-1}dF(z)=Vd\bar{z}+Udz=
		\begin{pmatrix} -\varphi_{\bar{z}} & 0 & e^{-\varphi}\bar{q} & 0 \\
				0 & \varphi_{\bar{z}} & 0 & e^{\varphi} \\ 
				0 & e^{-\varphi}\bar{q} & 0 & 0 \\
				e^{\varphi} & 0 & 0 & 0
		\end{pmatrix}d\bar{z}+ \begin{pmatrix} \varphi_{z} & 0 & 0 & e^{\varphi} \\
							0 & -\varphi_{z} & qe^{-\varphi} & 0 \\
							qe^{-\varphi} & 0 & 0 & 0 \\
							0 & e^{\varphi} & 0 & 0 
				       \end{pmatrix}dz \ .
\end{equation}
Now, we identify $(\tilde{S}^{reg}_{i}, \tilde{h}_{n})$ with a simply connected domain of $\h^{2}$ and we fix a lift $\tilde{p}_{n,i}$ of the base point. For every $\gamma \in \pi_{1}(S_{i}^{reg},p_{i})$, which we think of as a holomorphic automorphism of $\h^{2}$, the uniqueness of the solution to the initial value problem implies that $F_{n}=\gamma^{*}F_{n}$, where the frame field pulls back under $\gamma$ to
\[
 \gamma^{*}F_{n}(z)=\left\{\frac{\gamma'}{|\gamma'|}v_{1}(\gamma \cdot z), \frac{\overline{\gamma'}}{|\overline{\gamma'}|}v_{2}(\gamma \cdot z), N\circ \gamma, \tilde{\sigma}\circ \gamma \right\} \ .
\]
Let $D_{n,\gamma}$ be the diagonal matrix 
\[
	D_{n,\gamma}=\diag(\gamma'/|\gamma'|, \overline{\gamma'}/|\overline{\gamma'}|, 1 ,1) \ ,
\]
so that we can write 
\[
	\gamma^{*}F_{n}=H_{n,\gamma}D_{n,\gamma}
\]
where $H_{n,\gamma}$ is the matrix defined by 
\[
	H_{n,\gamma}: \{ v_{1}(\tilde{p}_{n,i}), v_{2}(\tilde{p}_{n,i}), N(\tilde{p}_{n,i}), \tilde{\sigma}_{n,i}(\tilde{p}_{n,i})\} \mapsto 
			\{ v_{1}(\gamma \cdot \tilde{p}_{n,i}), v_{2}(\gamma \cdot \tilde{p}_{n,i}), N(\gamma\cdot \tilde{p}_{n,i})), \tilde{\sigma}_{n,i}(\gamma \cdot \tilde{p}_{n,i})\} \ ,
\]
which can be obtained by solving Equation (\ref{eq:framefield}) with the identity as initial condition along the geodesic lift of $\gamma$ passing through $\tilde{p}_{n,i}$.
The matrix $\tilde{\rho}_{n,i}(\gamma)$ is then conjugated to $H_{n,\gamma}D_{n,\gamma}$, where the conjugating factor is the frame field of the immersion at the base point.

\begin{prop}\label{prop:holo} Let $\rho_{i,n}:\pi_{1}(S^{reg}_{i})\rightarrow \SO_{0}(2,2)$ be the sequence of representations associated to the equivariant embeddings $\tilde{\sigma}_{i,n}$ with respect to the base point $p_{i}$. Then $\rho_{n,i}$ converges to a representation $\rho_{\infty,i}:\pi_{1}(\Sigma^{reg}_{i})\rightarrow \SO_{0}(2,2)$.
\end{prop}
\begin{proof} In the above setting, let $\{\alpha_{j}\}_{j=1, \dots, k}$ be a set of generators of $\pi_{1}(S^{reg}_{i})$ based at $p_{i}$. Let $\tilde{\alpha}_{j,n}$ be the lift of $\alpha_{j}$ starting at $\tilde{p}_{i,n}$ and ending at $\alpha_{j} \cdot \tilde{p}_{i,n}$. We denote with $\hat{\alpha}_{j,n}$ the geodesic segments homotopic to $\tilde{\alpha}_{j,n}$ relative to the end points. Since $h_{n}$ converges to $h_{\infty}$ on the regular parts, we know that $\tilde{\alpha}_{j,n}$ converges to $\tilde{\alpha}_{j,\infty}$ smoothly. Therefore, using the notations introduced in the above discussion, the sequence of matrices $H_{n, \alpha_{j}}$ converges when $n$ goes to infinity because, as the loop $\alpha_{j}$ is entirely contained in a compact set of $S^{reg}_{i}$, where we have uniform bounds on the conformal factors $u_{n}$ and where the background hyperbolic metrics $h_{n}$ and the quadratic differentials $q_{n}$ converge uniformly to $h_{\infty}$ and $q_{\infty}$, the system of ODE has uniformly bounded coefficients, hence the solutions to the initial value problem converge. \\
As for the diagonal matrices $D_{n,\alpha_{j}}$, they clearly converge because by assumption the Fuchsian groups realizing $(S^{reg}_{i}, h_{n})$ as quotients of domains in the hyperbolic plane with totally geodesic boundary are converging to the Fuchsian group $\Gamma_{\infty,i}$ so that $(\Sigma^{reg}_{i}, h_{\infty,i})=\h^{2}/\Gamma_{\infty,i}$. Since this holds for a set of generators of $\pi_{1}(S^{reg})$, the whole sequence of representation converges.
\end{proof}

We notice, moreover, that by taking the limit in the relation
\[
	\tilde{\sigma}_{i,n}(\alpha\cdot \tilde{p})=\rho_{n}(\alpha)\tilde{\sigma}_{i,n}(\tilde{p})   \ \ \ \ \  \forall \alpha \in\pi_{1}(S^{reg}_{i}) \ \ \forall \tilde{p}\in \tilde{S}^{reg}_{i}
\]
we obtain that $\sigma_{\infty,i}$ is equivariant with respect to $\rho_{\infty}$.

\begin{proof}[Proof of Theorem A] We can realise $M_{n}$ as quotient of a globally hyperbolic domain $\Omega_{n}$ of $\AdS_{3}$ by $\rho_{n}(\pi_{1}(S))$. In the same way, we can realise the (possibly disconnected) regular GHM anti-de Sitter manifold diffeomorphic to $\Sigma^{reg}\times \R$ corresponding to $(h_{\infty}, q_{\infty})$ as a quotient of a collection of domains $\Omega_{\infty, i}$ (one for each connected component of $\Sigma^{reg}$) by the groups $\rho_{\infty}(\pi_{1}(\Sigma^{reg},p_{i}))$, because by Proposition \ref{prop:holo} and the subsequent remark, $\rho_{\infty}$ is exactly the holonomy representation of the maximal embedding.  We recall that the domains $\Omega_{n}$ and $\Omega_{\infty, i}$ are completely determined by the boundary at infinity of the maximal surface by the property that a point $z\in \Omega_{n}$ (or $\Omega_{\infty,i}$) if and only if its projective dual plane is disjoint from the boundary at infinity of the maximal surface.
By Propositon \ref{prop:embeddings}, after fixing base points $\tilde{p}_{n,i}$ and $\tilde{p}_{i}$, we find that $(\Omega_{n}, \tilde{\sigma}_{n}(\tilde{p}_{i,n}))$ converges to $(\Omega_{\infty,i}, \tilde{\sigma}_{\infty,i}(\tilde{p}_{i}))$ because smooth convergence of the maximal embeddings on compact sets implies convergence of their boundary at infinity and, hence, of the domains of discontinuity by the previous remark. Therefore, each connected component of $M_{\infty}$ is a possible geometric limit of $M_{n}$, as claimed.
\end{proof}

\bibliographystyle{alpha}
\bibliographystyle{ieeetr}
\bibliography{bs-bibliography}

\bigskip

\noindent \footnotesize \textsc{DEPARTMENT OF MATHEMATICS, RICE UNIVERSITY}\\
\emph{E-mail address:}  \verb|andrea_tamburelli@libero.it|

\end{document}